\documentclass[sn-mathphys]{sn-jnl}

\jyear{2021}%

\theoremstyle{thmstyleone}%
\newtheorem{theorem}{Theorem}[section]
\theoremstyle{thmstyletwo}%
\newtheorem{remark}{Remark}[section]%

\theoremstyle{thmstylethree}%
\newtheorem{definition}{Definition}[section]%
\newtheorem{lemma}{Lemma}[section]
\newtheorem{corollary}{Corollary}[section]
\begin{document}

\title{ Some extensions of Krasnoselskii's  fixed point result for real functions}


\author*[1]{\fnm{Hassan} \sur{Khandani}}\email{ Hassan.Khandani@iau.ac.ir, Khandani.hassan@yahoo.com}


\affil*[1]{\orgdiv{Department of mathematics, Mahabad Branch}, \orgname{Islamic Azad university}, \orgaddress{ \city{Mahabad}, \country{Iran}}}

\abstract{We extend Krasnoselskii's fixed point result to non-self-real functions. We find a new and simple proof for Hillam's result. In our approach, we don't assume the image of the related mapping to be compact or bounded. In this way, we extend Hillam's result to self-mappings on $\mathbb R$. Finally, we present a new proof for the global convergence of the Newton-Raphson method.}

\keywords{ Krasnoselskii's theorem, Iterative sequence, Lipschitz function, Fixed point, Real function}
\pacs[MSC Classification]{26A18, 49M15}
\maketitle
\section{Introduction}\label{sec1}
We start our discussion by presenting the definition of L-Lipschitz functions, Krasnoselskii's theorem, and some generalizations of this result. In these results, the related function is self-mapping. In this manuscript, we extend Theorem \ref{hillam} to non-self mappings that are more useful in practice.
\begin{definition}
Let $L>0$, $h:[a,b]\to \mathbb R$ ia called an L-Lipschitz function if $\lvert h(x)-h(y)\rvert\le L\lvert x-y\rvert$ for each $x,y\in [a,b].$
\end{definition}
\begin{definition}(Krasnoselskii's sequence \cite{berinde2007iterative})\label{krseq}
Let $X$ be a Banach space and $A$ be a nonempty subset of $X$. For each $x_0\in A$ and $t\in (0,1]$, the Krasnoselskii iteration of $h:A\to A$ is defined as follows:
\begin{equation}\label{krsequence}
x_{n+1}=t x_n+(1-t)h(x_n)\text{ for each } n\ge 0.
\end{equation}
\end{definition}
\begin{theorem}[Krasnoselskii \cite{krssnoselskiitwo}]
Suppose $A$ is a uniformly convex compact subset of a Banach space $X$. For every nonexpansive mapping $h:A\to A$ (i.e., $\lvert h(x)-h(y)\rvert\le \lvert x-y\rvert$ for each $x,y\in A$) the sequence of iterations defined by Equation \ref{krsequence} converges to a fixed point of $h$.
\end{theorem}
Moreover, Edelstein generalized this result to Banach spaces with a strictly convex norm \cite{edelstein1966remark}. Ishikawa showed this result holds for an arbitrary Banach space \cite{edelstein1978nonexpansive}. Edelstein and O'brien independently confirmed this result again \cite{kirk2000nonexpansive}.\\
Apart from all these generalizations, Bailey gave proof of Krasnoselskii's result for nonexpansive real-valued functions on a closed interval \cite{subrahmanyam2018elementary}, \cite{bailey1974krasnoselski}. For some other fixed point results about real-valued functions based on Krasnoselskii's method, we also refer the reader to \cite{borwein1991fixed}. Hillam B. P. \cite{hillam1975generalization} extended Bailey's result to Lipschitzian functions, which is the main subject of our discussion.
\begin{theorem}[Hillam \cite{hillam1975generalization}]\label{hillam}
Let $L>0$, and $h:[a,b]\to [a,b]$ be a $L$-Lipschitz mapping. Then, the sequence of iterations defined by Equation \ref{krsequence} converges monotonically to a fixed point of $h$ if $0<t\le \frac{1}{1+L}$.
\end{theorem}
As we mentioned before, in the proof of all Krasnoselskii fixed point generalizations, especially in Theorem \ref{hillam}, the fact that $h$ is a self-mapping on a closed interval plays a crucial role.\\
The structure of this manuscript is as follows: First, In Theorems \ref{th1} and \ref{th2} we replace the differentiability condition with Lipschitz condition and extend these results to some Krasnoselskii type results for non-self real-valued mappings, Lemmas \ref{existextension1} and \ref{existextension2} of Section \ref{sec2}. Then, we extend Hillam's theorem to self-mappings on $\mathbb R$, Theorem \ref{krass}. After that, we present this result for non-self mappings. Finally, we give a simple proof for Hillam's Theorem. In Section \ref{sec3}, we study conditions under which the Newton-Raphson sequence always converges for any starting point, which is called the global convergence of the Newton-Raphson method. The global convergence of the Newton-Raphson method has already been studied by J. L. Moriss \cite{morris1983computational}. In Section \ref{sec3}, based on Theorems \ref{th1} and \ref{th2} we present a more brief and simple proof for the global convergence of this method. L. Thorlund-Petersen characterized all the functions for which the Newton-Raphson method converges globally \cite{thorlund2004global}.\\
We need the following results in the sequel.\\
Khandani et al. \cite{khandani2021iterative} proved some results, similar to Theorem \ref{hillam}, for real differentiable functions and the Krasnoselskii's sequence $x_{n+1}=\frac{x_n+h(x_n)}{2}$ as follows:
\begin{theorem}[khandani, H. and Khojasteh. F \cite{khandani2021iterative}]\label{th1}
Let $h$ be a continuous real-valued function on $[a,c]$ that is differentiable on $(a,c)$ with $h^{'}(x)\ge -1$ on $(a,c)$, $h(x)>x$ for each $x\in[a,c)$, and $c$ is the unique fixed point of $h$ in $[a,c]$. Let $x_0\in [a,c)$ and for each $n\ge 0$ define:
\begin{equation}
x_{n+1}=\frac{x_n+h(x_n)}{2},
\end{equation}
then the sequence $\{x_n\}$ converges to $c$.
\end{theorem}
\begin{theorem}[khandani, H. and Khojasteh. F \cite{khandani2021iterative}]\label{th2}
Let $h$ be a continuous real-valued function on $[c,b]$ that is differentiable on $(c,b)$ with $h^{'}(x)\ge -1$ on $(c,b)$, $h(x)<x$ for each $x\in(c,b]$, and $c$ is the unique fixed point of $h$ in $[c,b]$. Let $x_0\in (c,b]$ and for each $n\ge 0$ define:
\begin{equation}\label{meanex2}
x_{n+1}=\frac{x_n+h(x_n)}{2},
\end{equation}
then the sequence $\{x_n\}$ converges to $c$.
\end{theorem}
In this manuscript, we denote the set of real numbers by $\mathbb R$. We denote the set $\{0,1,2,\dots \}$ of nonnegative integers by $\mathbb N$. For each $a,b\in \mathbb R$ with $a<b$, $[a,b]=\{x\in \mathbb R:a\le x\le b\}$ and $(a,b)=\{x\in \mathbb R:a< x< b\}$ are called closed and open interval from $a$ to $b$ respectively.
\section{Main results}\label{sec2}
First, in Theorems \ref{th1}, \ref{th2}, we replace the differentiability condition with the Lipschitz condition. Then, we replace related recursive sequences with the Krasnoselskii sequence. We produce two fixed-point results. Then, we provide a simple and new proof for Hillam's theorem. Therefore, the following two Lemmas that extend Krasnoselskii's result to non-self mappings are among the main results of this manuscript.
\begin{lemma}\label{existextension1}
Let $L>0$, $h:[a,c]\to \mathbb R$ be a $L$-Lipschitz mapping, $h(x)>x$ for each $x\in[a,c)$, and $c$ be the unique fixed point of $h$ in $[a,c]$. Let $x_0\in [a,c]$ and for each $n\ge 0$ define:
\begin{equation}\label{meanex}
x_{n+1}=(1-t) x_n+t h(x_n),
\end{equation}
then for each $0<t\le \frac{1}{1+L}$ the sequence $\{x_n\}$ converges to $c$.
\end{lemma}
\begin{proof} If $x_m=c$ for some non-negative integer $m$, then we have $x_n=c$ for each $n\ge m$ and the proof is complete. Therefore, we suppose $x_n\not = c$ for each $n\ge 0$. Let $0<t\le \frac{1}{1+L}$ and $x\in[a,c)$. Since $h$ is $L$-Lipschitz function we have $\frac{h(x)-c}{x-c}\ge -L\ge 1-\frac{1}{t}$. This follows that $\frac{h(x)-x}{x-c}\ge \frac{-1}{t}$. Rearranging this inequality, we get:
\begin{equation}\label{crse}
(1-t)(x)+t h(x)\le c
\end{equation}
Assume $\{x_n\}$ be defined by Equation \ref{meanex}. From equation \ref{crse} and argument by induction we see that $x_{n}< c$ for each $n\ge 0$. We show that $\{x_n\}$ is an increasing sequence. Since $h(x_0)>x_0$, we have $c>x_1=(1-t) x_0+t h(x_0)>x_0$. Suppose $x_0<x_1<\dots<x_m<c$. Since $a\le x_m<c$, by our assumption $h(x_m)>x_m$ and $x_{m+1}=(1-t) x_m+t h(x_m)>x_m$. Therefore, by induction, $\{x_n\}$ is an increasing sequence. Now, $x_n<c$ for each $n\ge 0$ and $\{x_n\}$ is an increasing sequence in $[a,c]$. Assume $x_n$ converges to $b\in[a,c]$. $b$ is also a fixed point of $h$. Since f has a unique fixed point in $[a,c]$, we deduce that $b=c$. Now, the proof is complete.
\end{proof}
\begin{lemma}\label{existextension2}
Let $L>0$, $h:[c,a]\to \mathbb R$ be a $L$-Lipschitz mapping, $h(x)<x$ for each $x\in(c,a]$, and $c$ be the unique fixed point of $h$ in $[c,a]$. Let $x_0\in [c,a]$ and for each $n\ge 0$ define:
\begin{equation}\label{meanex}
x_{n+1}=(1-t) x_n+t h(x_n),
\end{equation}
then for each $0<t\le \frac{1}{1+L}$ the sequence $\{x_n\}$ converges to $c$.
\end{lemma}
\begin{proof}
The proof is similar to that of Lemma \ref{existextension1} and therefore, we omit it.
\end{proof}
Now, in Hillam's fixed Theorem \cite{hillam1975generalization}, we replace the interval $[a,b]$ with $\mathbb R$ and extend this result as follows.
\begin{theorem}\label{krass}
Let $L>0$, $h: \mathbb R\to \mathbb R$ be a $L$-Lipschitz mapping. Let $e\in\mathbb R$, define $x_0=e$ and
\begin{equation}\label{mea}
x_{n+1}=(1-t) x_n+t h(x_n), n\ge 1
\end{equation}
Then, $\{x_n\}$ is a monotone sequence that either converges to a fixed point of $h$ or diverges to infinity.
\end{theorem}
\begin{proof}Assume $x_0=e\in [a,b]$. If $h(x_0)=x_0$, then $\{x_n\}$ is a constant sequence and converges to the fixed point $e=x_0$. If $h(x_0)>x_0$, then we have the following two different cases:
\begin{itemize}
\item (a) There exists a fixed point $d$ such that $x_0<d$. In this case, suppose that $c$ be the smallest fixed point of $h$ such that $x_0<c$. Then, $h:[x_0,c]\to \mathbb R$ satisfies all the conditions of \ref{existextension1} and therefore $\{x_n\}$ converges to $c$.
\item (b) There is no fixed point $d$ such that $x_0<d$. Suppose $f(y)<y$ for some $y>x_0$. Then, using intermediate value theorem for the continuous function $h(x)-x$ on the interval $[x_0,y]$ follows that there exists a point $z$ with $x_0<z<y$ such that $h(z)=z$. This is a contradiction, therefore, we deduce that $h(x)>x$ each $x>x_0$. We have $x_1=(1-t) x_0+t h(x_0)>(1-t) x_0+t x_0=x_0$. By induction, we deduce that $x_{n+1}=(1-t) x_n+t h(x_n)>x_n$ for each $n\ge 0$. Therefore, $\{x_n\}$ is an increasing sequence. If $\{x_n\}$ be a bounded sequence, then it converges to some point $c>x_0$. Since $h$ is a continuous map, taking limit of both sides of Equation \ref{mea} follows that $h(c)=c$ that is a contradiction. Therefore, $\{x_n\}$ is not bounded and diverges to infinity.
\end{itemize}
Therefore, the proof is complete when $f(x_0)>x_0$. A similar proof holds when $f(x_0)<x_0$. The only difference is that the sequence $\{x_n\}$ is decreasing when $f(x_0)<x_0$.
\end{proof}
We extend Theorem \ref{krass} as follows to non-self mapping. The only difference is that the sequence $\{x_n\}$ may exit the interval $[a,b]$ instead of diverging to infinity.
\begin{corollary}
Let $L,a,b>0,a<b$, $h: [a,b]\to \mathbb R$ be a $L$-Lipschitz mapping. Let $e\in\mathbb R$, define $x_0=e$ and
\begin{equation}
x_{n+1}=(1-t) x_n+t h(x_n), n\ge 1
\end{equation}
Then, $\{x_n\}$ is a monotone sequence that either converges to a fixed point of $h$ or exits the interval $[a,b]$.
\end{corollary}
\begin{proof}
The proof is the same as with Theorem \ref{krass}. Since the sequence $\{x_n\}$ is monotone, it exits the interval if there is no fixed point in front of it.
\end{proof}
Now, we deduce Theorem \ref{hillam} as follows:
\begin{theorem}[Hillam's Theorem]
Let $L>0$, then for every $L$-Lipschitz mapping mapping $h:[a,b]\to [a,b]$ the sequence of iterations defined by Equation \ref{krsequence} converges monotonically to a fixed point of $h$ if $0<t\le \frac{1}{1+L}$.
\end{theorem}
\begin{proof}Let $x_0\in [a,b]$ and define the sequence $\{x_n\}$ as Equation \ref{meanex}. First, suppose $h(x_0)>x_0$. Since $h(x)-x$ is a continuous function on $[x_0,b]$ and $h(b)\le b$, there exists a fixed point $c$ of $h$ such that $x_0\le c\le b$. Suppose $c$ be the smallest fixed point of $h$ such that $c>x_0$, $c$ is the unique fixed point of $h$ in $[a,c]$. Then, $h:[x_0,c]\to\mathbb R$ satisfies all conditions of Lemma \ref{existextension1}, so $\{x_n\}$ converges to $c$ that completes the proof. If $h(x_0)<x_0$, using Lemma \ref{existextension2}, the proof follows similarly.
\end{proof}
\begin{remark}
If $h:[a,b]\to \mathbb R$ be an L-Lipschitz function, then $\frac{h(x)-h(y)}{x-y}\ge-L$ for each $x,y\in [a,b]$. All our presented results hold if we just assume $\frac{h(x)-h(y)}{x-y}\ge-L$ for each $x,y\in [a,b]$ that is weaker than the L-Lipschitz condition.
\end{remark}
\section{A new proof of global convergence condition for the Newton-Raphson method}\label{sec3}
In this section, we show that the Newton-Raphson sequence is a special case of Krasnoselskii's sequence, and we provide new proof for its global convergence as follows. It is worth mentioning that these results have been presented already as a pre-print paper\cite{hassan}.
\begin{theorem}\label{convergence1}
Suppose that $a,c\in \mathbb R$ with $a<c$, $h:[a,c]\to \mathbb R$ is a real-valued function, $c$ is the unique root of $h$ in $[a,c]$. Also assume that $h^{''}(x),h^{'}(x)$ exist for each $x\in (a,c)$, $h(x)h^{''}(x)\ge 0$ for each $x\in (a,c)$, $h(x)h^{'}(x)< 0$ for each $x\in [a,c)$, $h^{'}(x)\not =0$ for each $x\in (a,c)$, $h^{'}(a+)\not =0$, $h^{'}(c-)\not =0$. For each $x_0\in [a,c]$ define:
\begin{equation}\label{newtonseq}
x_{n+1}=x_n-\frac{h(x_n)}{h^{'}(x_n)}\text { for all } n\ge 0.
\end{equation}
Then, $\{x_n\}$ converges to $c$ as $n\to \infty$.
\end{theorem}
\begin{proof}
For each $x\in [a,c]$, define $g(x)=x-\frac{2h(x)}{h^{'}(x)}$. By our assumptions $h(x)>x$ for each $x\in[a,c)$ and
\begin{equation}\label{newtonseq}
g^{'}(x)=-1+2\frac{h^{''}(x)h(x)}{(h^{'}(x))^2}\ge -1 \text{ for each } x\in(a,c).
\end{equation}
We have $\frac{g(x_n)+x_n}{2}=x_n-\frac{h(x_n)}{h^{'}(x_n)}=x_{n+1}$. Now, $g$ satisfies all conditions of Lemma $\ref{th1}$, so $\{x_n\}$ converges to $c$ as $n\to \infty$ where $c$ is the fixed point of $h$. We have $c=c-\frac{h(c)}{h^{'}(c-)}$ which follows that $h(c)=0.$
\end{proof}
\begin{theorem}\label{convergence2}
Suppose that $c,b\in \mathbb R$ with $c<b$, $h:[c,b]\to \mathbb R$ is a real-valued function, $c$ is the unique root of $h$ in $[c,b]$. Also assume that $h^{''}(x),h^{'}(x)$ exist for each $x\in (c,b)$, $h(x)h^{''}(x)\ge 0$ for each $x\in (c,b)$, $h(x)h^{'}(x)>0$ for each $x\in (c,b]$, $h^{'}(x)\not =0$ for each $x\in (c,b)$, $h^{'}(c+)\not =0$ and $h^{'}(b-)\not =0$. For each $x_0\in [c,b]$ define:
\begin{equation}\label{newtonseq}
x_{n+1}=x_n-\frac{h(x_n)}{h^{'}(x_n)}\text { for all } n\ge 0.
\end{equation}
Then, $\{x_n\}$ converges to $c$ as $n\to \infty$.
\end{theorem}
\begin{proof}
For each $x\in [c,b]$, define $g(x)=x-\frac{2h(x)}{h^{'}(x)}$. By our assumptions $h(x)<x$ for each $x\in(c,b]$ and
\begin{equation}\label{newtonseq}
g^{'}(x)=-1+2\frac{h^{''}(x)h(x)}{(h^{'}(x))^2}\ge -1 \text{ for each } x\in(c,b).
\end{equation}
We have $\frac{g(x_n)+x_n}{2}=x_n-\frac{h(x_n)}{h^{'}(x_n)}=x_{n+1}$. Now, $g$ satisfies all conditions of Lemma $\ref{th2}$, so $\{x_n\}$ converges to $c$ as $n\to \infty$ where $c$ is the fixed point of $h$. We have $c=c-\frac{h(c)}{h^{'}(c+)}$ which follows that $h(c)=0.$
\end{proof}
\section*{Conclusion}
In this paper, we introduced a new method for finding the fixed points of Lipschitz mappings. Using this method, we extended Krasnoselskii's fixed point theorem to non-self mappings. It is worth mentioning that generalizing a result to non-self mapping is more useful in practice and can be applied to more functions. We provided a new simple proof of Hillam's result. Our results are valuable for two reasons. First, using them, we can obtain an algorithm to estimate the fixed point or real root of a Lipschitz map. On the other hand, using this method, it is possible to give  new proofs of Krasnoselski's result for  self or non-self maps on $\mathbb R^n$ and to present new fixed point results.
We also investigated the convergence of the Newton-Raphson method and presented the global convergence condition for this method. This global convergence condition is not a new result, but we have obtained it more easily. Moreover, we can apply our method to more iterative sequences and provide convergence conditions for these iterations.
\section*{Statements and Declarations}
\begin{itemize}
\item (Author contribution) All the results of this manuscript are provided by  Khandani. H, the first and only author. The results presented in Section \ref{sec3} are from a  pre-print paper and can be accessed at  https://arxiv.org/abs/2112.04898v1. We have mentioned  this paper in the list of references \cite{hassan}.
\item (Funding)The authors did not receive support from any organization for the submitted work.
\item (Competing Interests)The authors have no competing interests to declare that are relevant to the content of this article.
\item (Data availability) Not applicable.
\end{itemize}
\bibliography{sn-bibliography}


\begin{thebibliography}{13}
\ifx \bisbn   \undefined \def \bisbn  #1{ISBN #1}\fi
\ifx \binits  \undefined \def \binits#1{#1}\fi
\ifx \bauthor  \undefined \def \bauthor#1{#1}\fi
\ifx \batitle  \undefined \def \batitle#1{#1}\fi
\ifx \bjtitle  \undefined \def \bjtitle#1{#1}\fi
\ifx \bvolume  \undefined \def \bvolume#1{\textbf{#1}}\fi
\ifx \byear  \undefined \def \byear#1{#1}\fi
\ifx \bissue  \undefined \def \bissue#1{#1}\fi
\ifx \bfpage  \undefined \def \bfpage#1{#1}\fi
\ifx \blpage  \undefined \def \blpage #1{#1}\fi
\ifx \burl  \undefined \def \burl#1{\textsf{#1}}\fi
\ifx \doiurl  \undefined \def \doiurl#1{\url{https://doi.org/#1}}\fi
\ifx \betal  \undefined \def \betal{\textit{et al.}}\fi
\ifx \binstitute  \undefined \def \binstitute#1{#1}\fi
\ifx \binstitutionaled  \undefined \def \binstitutionaled#1{#1}\fi
\ifx \bctitle  \undefined \def \bctitle#1{#1}\fi
\ifx \beditor  \undefined \def \beditor#1{#1}\fi
\ifx \bpublisher  \undefined \def \bpublisher#1{#1}\fi
\ifx \bbtitle  \undefined \def \bbtitle#1{#1}\fi
\ifx \bedition  \undefined \def \bedition#1{#1}\fi
\ifx \bseriesno  \undefined \def \bseriesno#1{#1}\fi
\ifx \blocation  \undefined \def \blocation#1{#1}\fi
\ifx \bsertitle  \undefined \def \bsertitle#1{#1}\fi
\ifx \bsnm \undefined \def \bsnm#1{#1}\fi
\ifx \bsuffix \undefined \def \bsuffix#1{#1}\fi
\ifx \bparticle \undefined \def \bparticle#1{#1}\fi
\ifx \barticle \undefined \def \barticle#1{#1}\fi
\bibcommenthead
\ifx \bconfdate \undefined \def \bconfdate #1{#1}\fi
\ifx \botherref \undefined \def \botherref #1{#1}\fi
\ifx \url \undefined \def \url#1{\textsf{#1}}\fi
\ifx \bchapter \undefined \def \bchapter#1{#1}\fi
\ifx \bbook \undefined \def \bbook#1{#1}\fi
\ifx \bcomment \undefined \def \bcomment#1{#1}\fi
\ifx \oauthor \undefined \def \oauthor#1{#1}\fi
\ifx \citeauthoryear \undefined \def \citeauthoryear#1{#1}\fi
\ifx \endbibitem  \undefined \def \endbibitem {}\fi
\ifx \bconflocation  \undefined \def \bconflocation#1{#1}\fi
\ifx \arxivurl  \undefined \def \arxivurl#1{\textsf{#1}}\fi
\csname PreBibitemsHook\endcsname

\bibitem{berinde2007iterative}
\begin{bbook}
\bauthor{\bsnm{Berinde}, \binits{V.}},
\bauthor{\bsnm{Takens}, \binits{F.}}:
\bbtitle{Iterative Approximation of Fixed Points}
vol. \bseriesno{1912}.
\bpublisher{Springer},
\blocation{Berlin}
(\byear{2007})
\end{bbook}
\endbibitem

\bibitem{krssnoselskiitwo}
\begin{barticle}
\bauthor{\bsnm{Krasnoselskii}, \binits{M.}}:
\batitle{Two observations about the method of successive approximations uspehi
  math}.
\bjtitle{Applied mathematics and computation}
\bvolume{10}(\bissue{1}),
\bfpage{123}--\blpage{127}
(\byear{1955})
\end{barticle}
\endbibitem

\bibitem{edelstein1966remark}
\begin{barticle}
\bauthor{\bsnm{Edelstein}, \binits{M.}}:
\batitle{A remark on a theorem of ma krasnoselski}.
\bjtitle{Amer. Math. Monthly}
\bvolume{73},
\bfpage{509}--\blpage{510}
(\byear{1966})
\end{barticle}
\endbibitem

\bibitem{edelstein1978nonexpansive}
\begin{barticle}
\bauthor{\bsnm{Edelstein}, \binits{M.}},
\bauthor{\bsnm{O'Brien}, \binits{R.C.}}:
\batitle{Nonexpansive mappings, asymptotic regularity and successive
  approximations}.
\bjtitle{Journal of the London Mathematical Society}
\bvolume{2}(\bissue{3}),
\bfpage{547}--\blpage{554}
(\byear{1978})
\end{barticle}
\endbibitem

\bibitem{kirk2000nonexpansive}
\begin{barticle}
\bauthor{\bsnm{Kirk}, \binits{W.}}:
\batitle{Nonexpansive mappings and asymptotic regularity}.
\bjtitle{Nonlinear Analysis: Theory, Methods \& Applications}
\bvolume{40}(\bissue{1-8}),
\bfpage{323}--\blpage{332}
(\byear{2000})
\end{barticle}
\endbibitem

\bibitem{subrahmanyam2018elementary}
\begin{bbook}
\bauthor{\bsnm{Subrahmanyam}, \binits{P.V.}}:
\bbtitle{Elementary Fixed Point Theorems}.
\bpublisher{Springer},
\blocation{Ebook}
(\byear{2018})
\end{bbook}
\endbibitem

\bibitem{bailey1974krasnoselski}
\begin{barticle}
\bauthor{\bsnm{Bailey}, \binits{D.}}:
\batitle{Krasnoselski's theorem on the real line}.
\bjtitle{The American Mathematical Monthly}
\bvolume{81}(\bissue{5}),
\bfpage{506}--\blpage{507}
(\byear{1974})
\end{barticle}
\endbibitem

\bibitem{borwein1991fixed}
\begin{barticle}
\bauthor{\bsnm{Borwein}, \binits{D.}},
\bauthor{\bsnm{Borwein}, \binits{J.}}:
\batitle{Fixed point iterations for real functions}.
\bjtitle{Journal of Mathematical Analysis and Applications}
\bvolume{157}(\bissue{1}),
\bfpage{112}--\blpage{126}
(\byear{1991})
\end{barticle}
\endbibitem

\bibitem{hillam1975generalization}
\begin{barticle}
\bauthor{\bsnm{Hillam}, \binits{B.P.}}:
\batitle{A generalization of krasnoselski's theorem on the real line}.
\bjtitle{Mathematics Magazine}
\bvolume{48}(\bissue{3}),
\bfpage{167}--\blpage{168}
(\byear{1975})
\end{barticle}
\endbibitem

\bibitem{morris1983computational}
\begin{botherref}
\oauthor{\bsnm{Morris}, \binits{J.L.}}:
Computational methods in elementary numerical analysis.
JOHN WILEY \& SONS, INC., 605 THIRD AVE., NEW YORK, NY 10158, USA, 1983, 416
(1983)
\end{botherref}
\endbibitem

\bibitem{thorlund2004global}
\begin{barticle}
\bauthor{\bsnm{Thorlund-Petersen}, \binits{L.}}:
\batitle{Global convergence of newton’s method on an interval}.
\bjtitle{Mathematical Methods of Operations Research}
\bvolume{59}(\bissue{1}),
\bfpage{91}--\blpage{110}
(\byear{2004})
\end{barticle}
\endbibitem

\bibitem{khandani2021iterative}
\begin{botherref}
\oauthor{\bsnm{Khandani}, \binits{H.}},
\oauthor{\bsnm{Khojasteh}, \binits{F.}}:
A new method for estimating the real roots of real differentiable functions.
arXiv preprint arXiv:2111.14460
(2021)
\end{botherref}
\endbibitem

\bibitem{hassan}
\begin{botherref}
\oauthor{\bsnm{Khandani}, \binits{H.}}:
A convergence condition for Newton-Raphson method.
Preprint at \url{https://arxiv.org/abs/2112.04898v1}
(2021)
\end{botherref}
\endbibitem

\end{thebibliography}

\end{document}